\documentclass[a4paper, twoside,12pt]{article}
\usepackage[english]{babel}
\usepackage{amsmath,amsthm}
\usepackage{amsfonts}
\usepackage{amssymb}
\usepackage{mathrsfs}
\newtheorem{thm}{Theorem}[section]

\newtheorem{lem}[thm]{Lemma}

\theoremstyle{definition}

\theoremstyle{remark}

\numberwithin{equation}{section}
\title{  Product formula for the one-dimensional $(k,a)$-generalized  Fourier  kernel.}
\author{B\'{e}chir Amri \\
\small bechiramri69@gmail.com, \\
\small Department of Mathematics, College of Sciences, Taibah University,\\\small
P.O. Box 30002 Al Madinah Al Munawarah, Saudi Arabia \\ \small  Department of Mathematics, Faculty of Sciences of Bizerte, University of Carthage,  Tunis, Tunisia.  }
\date{}

\begin{document}
\maketitle

\begin{abstract}
In this  paper, a product formula for the one-dimensional $(k,a)$-generalized  Fourier  kernel is given for  $k\geq0$, $a>0$ and $2k>a-1$,  extending  the special case of  \cite{MSM} when $a=\frac{2}{n}$, $n\in \mathbb{N}$.

 \end{abstract}
\section{ Introduction }
For a fixed reflection group associated  with a root system $R$  and  for a multiplicity function $k\geq 0$, the   $(k,a)$-deformed  harmonic oscillator is given by
 \begin{equation*}
   \Delta_{k,a}= \|x\|^{2-a}\Delta_{k}-\|x\|^a
 \end{equation*}
  where  $a > 0$ is  a   parameter  and  $\Delta_k$ is   the Dunkl Laplacian operator on  $\mathbb{R}^d$.  This operator gives rise to  the  semigroup
 \begin{equation*}
   \mathscr{J}_{a}(z) = \exp\left(\frac{z}{a}\Delta_{k,a}\right)
 \end{equation*}
 for $z \in \mathbb{C}$ such that $ Re(z) \geq0$, first featured and studied in \cite{BKO}, where the authors defined in $L^2(\mathbb{R}^d,|x|^{a-2}\upsilon_{k,a}(x)dx )$ an unitary operator called  the $(k, a)$-generalized Fourier transform
  \begin{equation*}\label{}
     \mathscr{F}_{k,a}= e^{i\frac{\pi}{2}\frac{ d+a-2+\sum_{\alpha\in R} k(\alpha)}{a}}\mathscr{J}_{k,a}(\frac{i\pi}{2})
  \end{equation*}
which  can be    expressed as integral  transform:
  \begin{equation*}\label{ }
     \mathscr{F}_{k,a}(f)( \xi)= c_{k,a}\int_{\mathbb{R}^N} B_{k,a}(\xi,x) f(x)|x|^{a-2} \prod_{\alpha\in R}|\langle x,\alpha\rangle|^{k(\alpha)}dx.
  \end{equation*}
 with certain constant $c_{k,a}$. In particular, the case $a = 2$ corresponds to     Dunkl transform. Formal expressions   for $ B_{k,a}$    have been derived in \cite{BKO} as a  series representation,
but these expressions are not very useful from the analytic point of view.
 \par  in the one   dimensional case  the kernel $B_{k,a}$  is  given  by
 \begin{equation}\label{BB}
     B_{k,a}(\lambda,x)= \mathcal{J}_{\frac{2k-1}{a}}\left(\frac{2}{a}|\lambda x|^{a/2}\right)+
 m_{k,a}\lambda x \mathcal{J}_{\frac{2k+1}{a}}\left(\frac{2}{a}|\lambda x|^{a/2}\right), \qquad \lambda,x\in \mathbb{R}
 \end{equation}
 where
$$m_{k,a}=e^{\frac{-i\pi}{a}}\frac{\Gamma\left(\frac{2k+a-1}{a} \right)}{a^{\frac{ 2}{a}}\Gamma\left(\frac{2k+a+1}{a} \right)} $$
and $\mathcal{J}_{\nu}$  is the normalized  Bessel function.
\begin{eqnarray}\label{Jj}
\mathcal{J}_{\nu}(z)= \Gamma(\nu+1)\left(\frac{z}{2}\right)^{-\nu} J_\nu(z)=\Gamma(\nu+1)\sum_{n=0}^\infty \frac{(-1)^n\left(\frac{z}{2}\right)^{2n}}{ n!\Gamma(\nu+n+1)}.
\end{eqnarray}
Restricting then  to one dimensional case,  one of the classic problems that arises  is to describe  the product two $B_{k,a}'s$   in a most convenient way that is
$$B_{k,a}( \lambda,x)B_{k,a}(\lambda, y)=\int B_{k,a}(\lambda,z) d\gamma_{x,y}^{k,a}(z)$$
with $\gamma_{x,y}^{k,a}$ are  measures on $\mathbb{R}$
  which are uniformly
bounded with respect to total variation norm.
This  formula was established in  \cite{MSM}  for $a=\frac{2}{n}$,  $n\in \mathbb{N}$. The author's approach makes use of the   well-known  Gegenbauer's addition theorem for the Bessel functions. Our purpose
here is to  extend the formula of  \cite{MSM}  to the case $a>0$. To be more precise, $\gamma_{x,y}^{k,a}$ will
be derived  in terms of the associated Legendre functions which involved in the    infinite  integral of   product of three Bessel functions of the first kind, due to Macdonal \cite{Mac},(see also, \cite{Wat}). Through it,  and via  Hankel transform theory we present  some formulas for integrals involving  Bessel functions or their product.
 \section{Main Results }
 In this section, we establish two  integral formulas, which are expressed as Hankel transform of associate Legendre functions.
 \par    Recalling first
 the Macdonal  integral, that when $x$ and y are positive,
  \begin{eqnarray}\label{3b}
   R_{\mu,\nu}(x,y,z)&=&\int_0^\infty J_{\nu}(xt)J_{\nu}(yt)J_{\mu}(zt)t^{1-\mu} dt\nonumber\\ &=&
    \left\{
      \begin{array}{ll}
        0,  \qquad\qquad \qquad\qquad\qquad\qquad \;\hbox{$z< |x-y|$;} \\
         \frac{(xy)^{\mu-1}\sin ^{\mu-\frac{1}{2}}\theta}{\sqrt{2\pi}z^{\mu}} P^{\frac{1}{2}-\mu}_{\nu-\frac{1}{2}}(\cos \theta),
\; \hbox{ \quad   $|x-y| < z< x+y$;}  \\
     \frac{e^{(\mu-\frac{1}{2})\pi i}\sin((\nu-\mu)\pi)(xy)^{\mu-1}\sinh ^{\mu-\frac{1}{2}} \theta }{(\frac{1}{2}\pi^3)^\frac{1}{2}z^\mu}
Q^{\frac{1}{2}-\mu}_{\nu-\frac{1}{2}}(\cosh \theta),   \; \quad \hbox{$ x+y<z$,}
      \end{array}
    \right.
  \end{eqnarray}
provided  $Re(\mu)>-\frac{1}{2}$, $Re(\nu)>-\frac{1}{2}$,
 and  where here we write
$x^2+y^2-z^2=2xy\cos \theta  $ if $|x-y| < z< x+y$
 and $   z^2-x^2-y^2=2xy\cosh \theta$ if $  x+y < z$ .
The associated  Legendre functions $P^\mu_\nu$ and $Q^\mu_\nu$   are  given in term of hypergeometric function by (see \cite{B}, p.122)
\begin{equation}\label{pppp}
 P^\mu_\nu(x)=\frac{1}{\Gamma(1-\mu)}\left(\frac{1+x}{1-x}\right)^{\frac{\mu}{2}}\; _2F_1\left(\nu+1,-\nu,1-\mu,\;\frac{1-x}{2}\right),\quad
-1<x\leq  1
\end{equation}
and
 \begin{equation}\label{QQQ}
 Q^\mu_\nu(x)= e^{\mu \pi i} \frac{\sqrt{\pi}\Gamma(\mu+\nu+1)(x^2-1)^{\frac{\mu}{2}}}{2^{\nu+1}x^{\mu+\nu+1}\Gamma(\nu+\frac{3}{2})}
 \;_2F_1\left(\frac{\mu+\nu}{2}+1,\frac{\mu+\nu+1}{2},\nu+\frac{3}{2},\;\frac{1}{x^2}\right), 1<x.
 \end{equation}
  It will be observed that  if  $\nu-\mu=n$  is a nonnegative  integer then
\begin{eqnarray*}
     R_{\mu,\nu}(x,y,z)&=&
    \left\{
      \begin{array}{ll}
\frac{2^{\frac{1}{2}-\mu}\Gamma(2\mu) n!}{\Gamma(\nu+\mu)\Gamma(\mu+\frac{1}{2})}\frac{(xy)^{\mu-1}\sin ^{2\mu-1}\phi}{\sqrt{2\pi}z^{\mu}}  C_n^\mu(\cos \theta),
\; \hbox{ $|x-y| < z< x+y$;} \\
        0,  \qquad\qquad \qquad\qquad\qquad\qquad\qquad\;\hbox{$z< |x-y|$ or $z> x+y$.} \\
          \end{array}
    \right.
\end{eqnarray*}
where $C_n^\nu$ is the Gegenbauer polynomial.
\par We shall now discuss   integral representations
which are to be associated with the  Hankel transform. It is a well-known fact from  the theory of  Hankel transform  (see  \cite{Tit}, Ch.8)  that if $f$ is an
integrable function on  $(0,+\infty)$  and of bounded variation in a neighborhood of $t>0$, then the following holds
$$ \int_0^{+\infty}\left\{\int_0^{+\infty}f(r)J_\alpha(rz)\sqrt{rz}\;dr\right\}J_\alpha(tz)\sqrt{tz}\;dz=\frac{f(t+0)+f(t-0)}{2},$$
where $\alpha> -\frac{1}{2}$. If  we take $\alpha=\mu$ and
$$ f(r)= J_{\nu}(xr)J_{\nu}(yr) r^{\frac{1}{2} -\mu}$$
with  $\nu> -\frac{1}{2}$ and  $ \frac{1}{2} <\mu< 2\nu+\frac{3}{2}$ ( which assert the integrability of $f$ ) then  we have
\begin{equation*}
  J_{\nu}(xt)J_{\nu}(yt) t^{ -\mu}= \int_{0}^{\infty}  R_{\mu,\nu}(x,y,z)J_{\mu}(zt)z dz.
\end{equation*}
The formula can be extended  to  $\mu> -\frac{1}{2}$ and $\nu> -\frac{1}{2}$ by the principle of analytic continuation.
Hence in view of (\ref{Jj}) it follows that
\begin{equation}\label{1}
 (xy)^\nu  t^{ 2(\nu-\mu)} \mathcal{J}_{\nu}(xt)\mathcal{J}_{\nu}(yt)= \frac{2^{2\nu-\mu}\Gamma^2(\nu+1)}{ \Gamma(\mu+1)} \int_{0}^{\infty}  R_{\mu,\nu}(x,y,z)\mathcal{J}_{\mu}(zt) z^{\mu+1} dz.
\end{equation}
Taking   $\alpha=\nu$ and
$$ f(r)= J_{\nu}(xr)J_{\mu}(yr) r^{\frac{1}{2} -\mu},$$
a similar argument proves  that
\begin{equation*}
  J_{\nu}(xt)J_{\mu}(yt) t^{ -\mu}= \int_{0}^{\infty}  R_{\mu,\nu}(x,z,y)J_{\nu}(zt)z dz.
\end{equation*}
with   $\nu> -\frac{1}{2}$ and  $\mu> -\frac{1}{2}$. From which we have
\begin{equation}\label{2}
 x^\nu y^\mu   \mathcal{J}_{\nu}(xt)\mathcal{J}_{\mu}(yt)=  2^{\mu}\Gamma(\mu+1)  \int_{0}^{\infty}  R_{\mu,\nu}(x,z,y)\mathcal{J}_{\nu}(zt) z^{\nu+1} dz.
\end{equation}
\par
Let us now   consider the product   $B_{k,a}(\lambda,x)B_{k,a}(\lambda,y)$  which  in virtue of (\ref{BB}) is equal to
\begin{eqnarray}\label{PP}
 && \mathcal{J}_{\frac{2k-1}{a}}\left(\frac{2}{a}|\lambda x|^{a/2}\right) \mathcal{J}_{\frac{2k-1}{a}}\left(\frac{2}{a}|\lambda y|^{a/2}\right)\nonumber
\\&&\qquad\qquad\qquad\qquad\qquad+m_{k,a}^2  \lambda^2xy\mathcal{J}_{\frac{2k+1}{a}}\left(\frac{2}{a}|\lambda x|^{a/2}\right)\mathcal{J}_{\frac{2k+1}{a}}\left(\frac{2}{a}|\lambda y|^{a/2}\right)\nonumber
\\&&+m_{k,a} \lambda x  \mathcal{J}_{\frac{2k+1}{a}}\left(\frac{2}{a}|\lambda x|^{a/2}\right)
 \mathcal{J}_{\frac{2k-1}{a}}\left(\frac{2}{a}|\lambda y|^{a/2}\right)\nonumber
 \\&&\qquad\qquad\qquad \qquad\qquad+m_{k,a}  \lambda y
 \mathcal{J}_{\frac{2k-1}{a}}\left(\frac{2}{a}|\lambda x|^{a/2}\right)\mathcal{J}_{\frac{2k+1}{a}}\left(\frac{2}{a}|\lambda y|^{a/2}\right).
\end{eqnarray}
If we make use (\ref{1}) with $\mu=\nu=\frac{2k-1}{a} $ and $t=\frac{2}{a}|\lambda|^{\frac{a}{2}}$  the first  product of two Bessel functions in (\ref{PP})   may    be  written  as (for $x\neq 0$, $y\neq 0$)
\begin{eqnarray*}
 &&\mathcal{J}_{\frac{2k-1}{a}}\left(\frac{2}{a}|\lambda x|^{a/2}\right) \mathcal{J}_{\frac{2k-1}{a}}\left(\frac{2}{a}|\lambda y|^{a/2}\right)
\\&& = \frac{2^{ \frac{2k-1}{a} }\Gamma(\frac{2k-1}{a}+1)}{|xy|^{k-\frac{1}{2}}}\int_0^\infty R_{\frac{2k-1}{a},\frac{2k-1}{a}}(|x|^{\frac{a}{2}},|y|^{\frac{a}{2}},z)\mathcal{J}_{\frac{2k-1}{a}}\left(\frac{2}{a}|\lambda| ^{a/2}z\right)z^{\frac{2k-1}{a}+1}\; dz
\\ &&= a2^{ \frac{2k-1}{a}-1 }\Gamma\left(\frac{2k-1}{a}+1\right)\int_0^\infty \frac{R_{\frac{2k-1}{a},\frac{2k-1}{a}}(|x|^{\frac{a}{2}},|y|^{\frac{a}{2}},z^{\frac{a}{2}})}{(|xy|z)^{k-\frac{1}{2}}}\mathcal{J}_{\frac{2k-1}{a}}\left(\frac{2}{a}|\lambda| ^{a/2}z^{\frac{a}{2}}\right)z^{ 2k+a-2}\; dz
\\&&= a2^{ \frac{2k-1}{a}-2}\Gamma\left(\frac{2k-1}{a}+1\right)\int_{-\infty}^\infty \frac{R_{\frac{2k-1}{a},\frac{2k-1}{a}}(|x|^{\frac{a}{2}},|y|^{\frac{a}{2}},|z|^{\frac{a}{2}})}{|xyz|^{k-\frac{1}{2}}} B_{k,a}(\lambda,z) |z|^{ 2k+a-2}\; dz.
\end{eqnarray*}
Using  (\ref{1})   with  $\nu=\frac{2k+1}{a} $ and $\mu= \frac{2k-1}{a} $  the second product in (\ref{PP}) can also be written  as
\begin{eqnarray*}
 &&  m_{k,a}^2 \lambda^{ 2}xy \mathcal{J}_{\frac{2k+1}{a}}\left(\frac{2}{a}|\lambda x|^{a/2}\right) \mathcal{J}_{\frac{2k+1}{a}}\left(\frac{2}{a}|\lambda y|^{a/2}\right)=m_{k,a}^2\frac{2^{ \frac{2k-1}{a}} a^{\frac{4}{a}}\Gamma^2( \frac{2k+1}{a}+1)}{\Gamma( \frac{2k-1}{a}+1)}\\&&\qquad\qquad\quad\times   \int_0^{+\infty} sgn(xy)\frac{R_{\frac{2k-1}{a},\frac{2k+1}{a}}(|x|^{\frac{a}{2}},|y|^{\frac{a}{2}},z^{\frac{a}{2}})}{|xyz|^{k-\frac{1}{2}}}\mathcal{J}_{\frac{2k-1}{a}}\left(\frac{2}{a}|\lambda| ^{a/2}z^{\frac{a}{2}}\right)z^{ 2k+a-2}\; dz
\\&&\qquad\qquad\qquad\qquad\qquad\qquad\qquad\qquad\qquad\qquad= e^{\frac{-2i\pi}{a}}a 2^{ \frac{2k-1}{a}-2}\Gamma\left(\frac{2k-1}{a}+1\right)\\&&\qquad\qquad\qquad\qquad\times\int_{-\infty}^{+\infty} sgn(xy)\frac{R_{\frac{2k+1}{a},\frac{2k+1}{a}}(|x|^{\frac{a}{2}},|y|^{\frac{a}{2}},|z|^{\frac{a}{2}})}{|xyz|^{k-\frac{1}{2}}} B_{k,a}(\lambda,z) |z|^{ 2k+a-2}\; dz.
\end{eqnarray*}
 Applying now in the same manner (\ref2) with $v=\frac{2k+1}{a}$ and $\mu= \frac{2k-1}{a}$ we obtain  that
\begin{eqnarray*}
 && m_{k,a} \lambda x \mathcal{J}_{\frac{2k+1}{a}}\left(\frac{2}{a}|\lambda x|^{a/2}\right) \mathcal{J}_{\frac{2k-1}{a}}\left(\frac{2}{a}|\lambda y|^{a/2}\right)=
  a 2^{ \frac{2k-1}{a}-1}\Gamma\left(\frac{2k-1}{a}+1\right) m_{k,a} \\&&\qquad\qquad\times \int_0^{+\infty} sgn(x)\frac{R_{\frac{2k-1}{a},\frac{2k+1}{a}}(|x|^{\frac{a}{2}},|z|^{\frac{a}{2}},|y|^{\frac{a}{2}})}{(|xy|z)^{k-\frac{1}{2}}}\lambda z\mathcal{J}_{\frac{2k+1}{a}}\left(\frac{2}{a}|\lambda| ^{a/2}z\right)z^{ 2k+a-2}\; dz
\\&& \qquad\qquad\qquad\qquad\qquad\qquad \qquad\qquad\qquad\quad=a 2^{ \frac{2k-1}{a}-2}\Gamma\left(\frac{2k-1}{a}+1\right)\\&&
\qquad\qquad\qquad\qquad\quad\times \int_{-\infty}^{+\infty} \frac{sgn(xz)R_{\frac{2k-1}{a},\frac{2k+1}{a}}(|x|^{\frac{a}{2}},|z|^{\frac{a}{2}},|y|^{\frac{a}{2}})}{|xyz|^{k-\frac{1}{2}}} B_{k,a}(\lambda,z)
 |z|^{ 2k+a-2}\; dz
\end{eqnarray*}
 and
\begin{eqnarray*}
 && m_{k,a} |\lambda|y \mathcal{J}_{\frac{2k+1}{a}}\left(\frac{2}{a}|\lambda y|^{a/2}\right) \mathcal{J}_{\frac{2k-1}{a}}\left(\frac{2}{a}|\lambda x|^{a/2}\right)
 = a 2^{ \frac{2k-1}{a}-2}\Gamma\left(\frac{2k-1}{a}+1\right)\\&&\qquad\qquad\qquad\times\int_{-\infty}^{+\infty} sgn(yz)\frac{R_{\frac{2k-1}{a},\frac{2k+1}{a}}(|y|^{\frac{a}{2}},|z|^{\frac{a}{2}},|x|^{\frac{a}{2}})}{|xyz|^{k-\frac{1}{2}}} B_{k,a}(\lambda,z) |z|^{ 2k+a-2}\; dz.
\end{eqnarray*}
We are thus led to the  formula
\begin{equation}\label{pr}
   B_{k,a}(\lambda,x) B_{k,a}(\lambda,y)= \int_{-\infty}^{+\infty} B_{k,a}(\lambda,z)\Delta_{k,a}(x,y,z) |z|^{ 2k+a-2}\; dz
\end{equation}
 where
\begin{eqnarray*}\label{12}
&&  \Delta_{k,a}(x,y,z)= a 2^{ \frac{2k-1}{a}-2}\Gamma\left(\frac{2k-1}{a}+1\right)\\&&\times \Bigg\{ \frac{R_{\frac{2k-1}{a},\frac{2k-1}{a}}(|x|^{\frac{a}{2}},|y|^{\frac{a}{2}},|z|^{\frac{a}{2}})}{|xyz|^{k-\frac{1}{2}}}
+   e^{\frac{-2i\pi}{a}}sgn(xy)\frac{R_{\frac{2k-1}{a},\frac{2k+1}{a}}(|x|^{\frac{a}{2}},|y|^{\frac{a}{2}},|z|^{\frac{a}{2}})}{|xyz|^{k-\frac{1}{2}}}
  \\&& + sgn(xz)\frac{R_{\frac{2k-1}{a},\frac{2k+1}{a}}(|x|^{\frac{a}{2}},|z|^{\frac{a}{2}},|y|^{\frac{a}{2}})}{|xyz|^{k-\frac{1}{2}}}
 +sgn(yz)\frac{R_{\frac{2k-1}{a},\frac{2k+1}{a}}(|y|^{\frac{a}{2}},|z|^{\frac{a}{2}},|x|^{\frac{a}{2}})}{|xyz|^{k-\frac{1}{2}}}\Bigg\}.
\end{eqnarray*}
\begin{lem}\label{l1}
Let  $\mu>-\frac{1}{2}$ and $\nu>-\frac{1}{2}$. As variables $x>0$ and $y>0$ the integral
 $$\int_{0}^{+\infty}\frac{|R_{ \mu,\nu}( x,y, z)|}{(xy)^{ \mu}}\;z^{\mu+1}dz$$
is uniformly bounded.
\end{lem}
\begin{proof}
  The proof is based on  the integrals  of  \cite{WM1}    that  appeared in  (16) of 18.1 and  (23) and of 18.2, to get the following
\begin{eqnarray}\label{P}
 \int_{-1}^{1}(1-t^2)^{\frac{\mu}{2}-\frac{1}{4}}P_{\nu-\frac{1}{2}}^{\frac{1}{2}-\mu}(t)\;dt= \frac{\pi2^{\frac{1}{2}-\mu}\Gamma(\mu+ \frac{1}{2}) }{\left(\Gamma(\frac{\mu+\nu+1}{2})\right)^2 \Gamma(\frac{\mu-\nu+2}{2}) \Gamma(\frac{\mu-\nu+1}{2}) },
\end{eqnarray}
and
\begin{eqnarray}\label{Q}
 \int_{1}^+{\infty}(t^2-1)^{\frac{\mu}{2}-\frac{1}{4}}Q_{\nu-\frac{1}{2}}^{\frac{1}{2}-\mu}(t)\;dt=\sqrt{2} e^{i(\frac{1}{2}-\mu\pi)}
 \frac{\Gamma(\frac{1+\nu-\mu}{2})\Gamma(\frac{\nu-\mu}{2}+\frac{1}{4})\Gamma(\mu+\frac{3}{4})\Gamma(\frac{3}{4})}{\Gamma(\nu+\mu)\Gamma(\nu+\mu+1)}.
\end{eqnarray}
From (\ref{3b}) we have
 \begin{eqnarray*}
\int_{x+y}^{+\infty}\frac{|R_{ \mu,\nu}( x,y, z)|}{(xy)^{ \mu}}z^{\mu+1}dz=  \frac{ |\sin((\nu-\mu)\pi)|}{(\frac{1}{2}\pi^3)^\frac{1}{2}} \int_{x+y}^{+\infty}\frac{ \sinh ^{\mu-\frac{1}{2}} \theta }{xy }
Q^{\frac{1}{2}-\mu}_{\nu-\frac{1}{2}}(\cosh \theta)z dz.
\end{eqnarray*}
Putting the change of variable
$$t=\cosh \theta=\frac{z^2-x^2-y^2}{2xy}, $$
it follows that
\begin{eqnarray}\label{QQ}
\int_{x+y}^{+\infty}\frac{|R_{ \mu,\nu}( x,y, z)|}{(xy)^{ \mu}}z^{\mu+1}dz=  \frac{ |\sin((\nu-\mu)\pi)|}{(\frac{1}{2}\pi^3)^\frac{1}{2}} \int_{1}^{+\infty}
 (t^2-1)^{\frac{\mu}{2}-\frac{1}{4}}Q_{\nu-\frac{1}{2}}^{\frac{1}{2}-\mu}(t)\;dt.
\end{eqnarray}
 Similarly
  \begin{eqnarray*}
\int_{|x-y|}^{x+y}\frac{|R_{ \mu,\nu}( x,y, z)|}{(xy)^{ \mu}}z^{\mu+1}dz= \frac{1}{\sqrt{2\pi}} \int_{-1}^{1} (1-t^2)^{\frac{\mu}{2}-\frac{1}{4}}
 | P^{\frac{1}{2}-\mu}_{\nu-\frac{1}{2}}( t )|\;dt.
\end{eqnarray*}
In view of (\ref{pppp}) we see that    $P^{\frac{1}{2}-\mu}_{\nu-\frac{1}{2}}( t )\geq 0$ when $-\frac{1}{2}<\nu\leq \frac{1}{2}$ . Thus using (\ref{P}) together  with the contiguous relation (see 4.3.3 of \cite{WM}),
$$ P_{\nu+1}^\mu(t)= tP_{\nu}^\mu(t)-(\mu+\nu)(1-t^2)^{\frac{1}{2}}P_{\nu}^{\mu-1}(t)$$
 one can see that
$$\int_{|x-y|}^{x+y}\frac{|R_{ \mu,\nu}( x,y, z)|}{(xy)^{ \mu}}z^{\mu+1}dz$$
is uniformly bounded. Then combine this with (\ref{QQ}) and (\ref{Q}) to achive  the proof of the lemma.
\end{proof}
\begin{lem}\label{l2}
For $\mu>-\frac{1}{2}$ and $\nu>-\frac{1}{2}$ the integral
 $$\int_{0}^{+\infty}\frac{|R_{ \mu,\nu}( x,z, y)|}{(xy)^{ \mu}}\;z^{\mu+1}dz$$
is uniformly bounded  with respect to $x>0$ and $y>0$.
\end{lem}
\begin{proof} Let us denote by
$$I_1(x,y)=\int_{|x-y|}^{x+y}\frac{|R_{ \mu,\nu}( x,z, y)|}{(xy)^{ \mu}}\;z^{\mu+1}dz
\quad\text{and}\quad I_2(x,y)=\int_{x+y}^{\infty}\frac{|R_{ \mu,\nu}( x,z, y)|}{(xy)^{ \mu}}\;z^{\mu+1}dz.$$
We are therefore led to prove that $I_1(x,y)$ and $I_2(x,y)$ are bounded.
It is convenient to divide the proof into two cases  $x\geq y$  and   $x < y$. We use the letter $C$   to denote positive constant whose value  can change at each occurrence.
 \par Let us begin with the case $x\geq y$ where  we have  $I_2(x,y)=0$.  To establish the boundedness of $I_1$   we use the following identity
  \begin{eqnarray}
    \Gamma(1-\mu) P_{\nu}^\mu(t)=2^\mu(1-t^2)^{-\frac{\mu}{2}}\;_2F_1 \left( \frac{1+\nu-\mu}{2}, \frac{-\mu-\nu}{2}, 1-\mu,1-t^2 \right)
  \end{eqnarray}
   which follows   from well known properties of the hypergeometric function  $_2F_1$ (see also \cite{WM}, p.167).
In addition  the function $\;_2F_1 \left( \frac{1+\nu-\mu}{2}, \frac{-\mu-\nu}{2}, 1-\mu,1-t^2 \right)$ is bounded
when $0<t<1$. It is then clear that
\begin{equation}\label{ppp}
  |P_{\nu}^\mu(t)|\leq C \;(1-t^2)^{-\frac{\mu}{2}},\qquad 0\leq t\leq 1.
\end{equation}
Now  using (\ref{ppp}), we get
when $|x-z|\leq y \leq x+z$ ( which is also equivalent to  $x-y\leq z \leq x+y$),
$$\frac{|R_{ \mu,\nu}( x,z, y)|}{(xy)^\mu}\leq C\; \frac{z^{\mu-1}}{  x y^{2\mu}} \left\{1- \left(\frac{x^2+z^2-y^2}{2xz}\right)^2\right\}^{\mu-\frac{1}{2}}.$$
For convenience, we write
 $$ 1- \left(\frac{x^2+z^2-y^2}{2xz}\right)^2= \frac{((x+y)^2-z^2)(z^2-(x-y)^2)}{4(xz)^2}.$$
  Hence,
  $$\frac{|R_{ \mu,\nu}( x,z, y)|}{(xy)^\mu}\leq C \; \frac{\Big\{((x+y)^2-z^2)(z^2-(x-y)^2\Big\}^{\mu-\frac{1}{2}}}{(xyz)^{2\mu}}\;z^\mu=CW(x,y,z)z^{\mu}.$$
  Now observe that
  $$ \int_{x-y}^{x+y}W(x,y,z)z^{2\mu+1}\;dz=\frac{2^{2\mu-1}\sqrt{\pi}\Gamma(\mu+\frac{1}{2})}{\Gamma(\mu+1)}$$
and therefore we conclude that  $I_1(x,y)$ is   bounded.  Consider now    $y\geq x$.   We shall use  the following estimates that follows from   (\ref{pppp})  and  15.4(ii) of  \cite{Nist},
\begin{eqnarray}
|P_\nu^\mu(t)|&\leq &C (1-t^2)^{-\frac{\mu}{2} }, \qquad\text{if }\qquad     \mu >0,\label{11}\\
 |P_\nu^\mu(t)|&\leq& C (1-t^2)^{\frac{\mu}{2}},\qquad\text{if }\qquad   \mu <0,\label{22}\\
|P_\nu^\mu(t)|&\leq& C |\ln(e(1+t))| , \qquad\text{if } \qquad     \mu =0,\label{33}
\end{eqnarray}
where $-1<t<1$. Noting first that in view of (\ref{11}) and  (\ref{ppp}) one can  conclude the boundedness of $I_1$ for  $ \mu <\frac{1}{2}$   in a similar manner as before.
When $ \mu > \frac{1}{2}$  and  from (\ref{22}) we have  for $y-x<z<x+y$,
$$\frac{|R_{ \mu,\nu}( x,z, y)|}{(xy)^\mu}\leq C \;\frac{z^{\mu-1}}{xy^{2\mu}}$$
 and  thus,
 $$I_1(x,y)=\int_{y-x}^{x+y}\frac{|R_{ \mu,\nu}( x,z, y)|}{(xy)^{ \mu}}\;z^{\mu+1}dz\leq C\; \frac{(x+y)^{2\mu+1}-(y-x)^{2\mu+1}}{xy^{2\mu}}$$
 $$\leq C\; \frac{(x/y+1)^{2\mu+1}-(1-x/y)^{2\mu+1}}{ x/y}\leq C.$$
Since the function $\Big(t+1)^{2\mu+1}-(1-t)^{2\mu+1}\Big)t^{-1}$ is bounded on $(0,1)$. In the case $\mu=\frac{1}{2}$, the estimation of (\ref{33}) gives $$|I_1(x,y)|\leq \frac{C}{xy}\int_{y-x}^{x+y}\left(1+\ln\left(1+\frac{x^2+z^2-y^2}{2xz}\right)\right)zdz$$
 Using the Change of  variable
$$t=\frac{x^2+z^2-y^2}{2xz},$$
 one can see that
$$\frac{1}{xy}\int_{y-x}^{x+y}\ln\left(1+\frac{x^2+z^2-y^2}{2xz}\right)zdz\leq2 \int_{-1}^1 \frac{\ln(1+t)}{|t|}\;dt.$$
As a consequence $I_1$ is bounded.
We  come now to the boundedness of $I_2$.  According with (\ref{QQQ}) and  15.4(ii) of  \cite{Nist} we  get
\begin{eqnarray}
 |Q_\nu^\mu(t)|&\leq& C \frac{(t^2-1)^{-\frac{\mu}{2} }}{t^{\nu-\mu+1}},,\qquad\text{if }\qquad   \mu >0,\label{55}\\
|Q_\nu^\mu(t)|&\leq &C \frac{(t^2-1)^{\frac{\mu}{2} }}{t^{\nu+\mu+1}}, \qquad\text{if }\qquad     \mu <0,\label{44}\\
|Q_\nu^\mu(t)|&\leq& C \frac{(t^2-1)^{\frac{\mu}{2} }}{t^{\mu+\nu+1}}|\ln(1-t^{-2})| , \qquad\text{if } \qquad     \mu =0. \label{66}
\end{eqnarray}
If $\mu  > \frac{1}{2}$ then under  consideration  (\ref{44})   with  (\ref{3b}) we have
$$\frac{|R_{ \mu,\nu}( x,z, y)|}{(xy)^\mu}\leq C x^{\nu-\mu}y^{-2\mu}(y^2-x^2-z^2)^{\mu-\nu-1}z^{\nu}$$
 and
\begin{eqnarray*}
| I_2(x,y)|&\leq C& x^{\nu-\mu}y^{-2\mu} \int_{0}^{y-x} \frac{z^{\mu+\nu+1}}{(y^2-x^2-z^2)^{\nu-\mu+1}}\; dz
 \\&\leq&Cx^{\nu-\mu}y^{-2\mu}(y^2-x^2)^{\frac{3\mu-\nu}{2}} \int_{0}^{ \sqrt{\frac{y-x}{y+x}}} \frac{z^{\mu+\nu+1}}{(1-z^2)^{\nu-\mu+1}}\; dz
 \\&\leq& C \Psi(x/y),
\end{eqnarray*}
where
$$\Psi(t)=t^{\nu-\mu}(1-t^2)^{\frac{3\mu-\nu}{2}} \int_{0}^{ \sqrt{\frac{1-t}{1+t}}} \frac{z^{\mu+\nu+1}}{(1-z^2)^{\nu-\mu+1}}\; dz.$$
It  not hard  to verify that $\Psi$ is bounded on $(0,1)$, which implies that $I_2$ is bounded.\\
If $\mu<\frac{1}{2}$ then
$$ |I_2(x,y)|\leq \frac{C}{xy^{2\mu}}\int_0^{y-x}\frac{\left\{\left(\frac{y^2-x^2-z^2}{2xz}\right)^2-1\right\}^{\mu-\frac{1}{2}}}{\left(\frac{y^2-x^2-z^2}{2xz}\right)^{\nu+\mu}}
z^{2\mu}dz$$
letting  the  change of variable
$$t=\frac{y^2-x^2-z^2}{2xz},$$
it becomes
$$ |I_2(x,y)|\leq C\;y^{-2\mu}\int_1^{+\infty } \frac{(t^2-1)^{\mu-\frac{1}{2}}}{t^{\nu+\mu}}\;\frac{(\sqrt{x^2t^2+y^2-x^2}-xt)^{2\mu+1}}{\sqrt{x^2t^2+y^2-x^2}}\; dt .$$
As $y>x$
  $$  \frac{(\sqrt{x^2t^2+y^2-x^2}-xt)^{2\mu+1}}{\sqrt{x^2t^2+y^2-x^2}}\leq \left(\frac{y^2-x^2}{y} \right)^{2\mu+1}\leq y^{2\mu},$$
it follows that
$$ |I_2(x,y)|\leq C\; \int_1^{+\infty} \frac{(t^2-1)^{\mu-\frac{1}{2}}}{t^{\nu+\mu}} \; dt .$$
Similarly,  when  $\mu=\frac{1}{2}$ where  it follows from (\ref{66}) that
$$ |I_2(x,y)|\leq C\; \int_1^{+\infty} \frac{ \ln(1-t^{-2})}{t^{\nu+1/2}} \; dt .$$
  Consequently, the boundedness  of $I_2$ follows. This completes the proof of the lemma.
\end{proof}
  \par Now our main result can be stated as follows.
\begin{thm}\label{th1}
In one dimentional case the kernel
$ B_{k,a}$ satisfies  the product formula
$$ B_{k,a}(\lambda,x) B_{k,a}(\lambda,y)= \int_{-\infty}^{+\infty} B_{k,a}(\lambda,z)d\gamma_{x,y}^{k,a}(z)$$
where
$$d\gamma_{x,y}^{k,a}(z)=\left\{
  \begin{array}{ll}
     \Delta_{k,a}(x,y,z) |z|^{ 2k+a-2}dz, & \hbox{if $xy\neq 0$;} \\
   \delta_x (z) , & \hbox{if $y=0$  ;} \\
     \delta_y(z) & \hbox{if $x=0$.}
  \end{array}
\right.$$
Further for all $x,y\in \mathbb{R}$  the integral $$ \int_{-\infty}^{+\infty}|d\gamma_{x,y}^{k,a}(z) |$$
   is finite  and  uniformly bounded.
\end{thm}
  Note here  that the measure $\delta_{x,y}^{k,a}$  has compact  support  if and only if $a=\frac{2}{n}$, $n\in \mathbb{N}$.
 Next we define a similar measure $\sigma_{x,y}$  as 
 $$d \sigma_{x,y}^{k,a}(z)=\left\{
  \begin{array}{ll}
     \Delta_{k,a}(x,z,y) |z|^{ 2k+a-2}dz, & \hbox{if $xy\neq 0$;} \\
   \delta_x (z) , & \hbox{if $y=0$  ;} \\
     \delta_y(z) & \hbox{if $x=0$.}
  \end{array}
\right.$$
Then one can  use Lemmas  \ref{l1} and \ref{l2} to  get that  
$$ \int_{-\infty}^{+\infty}|d\gamma_{x,y}^{k,a}(z) |$$
is finite  and   uniformly bounded. The second main result conserned  with  the generalized translation operator $\tau^{k,a}_y$,  $y\in  \mathbb{R}$ which can be
  defined on $L^2(\mathbb{R},|x|^{2k+a-2})$  using the $(k,a)$-generalized  Fourier  by
$$\mathcal{F}_{k,a}(\tau^{k,a}_y(f))(x)= B_{k,a}(x,y)\mathcal{F}_{k,a} (f)(x),$$
(see \cite{Luc}).
By  Theorem \ref{th1}  we can write for compactly supported function $f$ and   $y\neq 0$,
\begin{eqnarray*}
 \mathcal{F}_{k,a}(\tau_y(f))(x)&=& c_{k,a} \int_{-\infty}^{+\infty}B_{k,a}(x,y)B_{k,a}(x,\xi)f(\xi)|\xi|^{2k+a-2}d\xi
\\&=&c_{k,a} \int_{-\infty}^{+\infty} \int_{-\infty}^{+\infty} B_{k,a}(x,z)\Delta_{k,a}(y,\xi,z)f(\xi)|\xi|^{2k+a-2} |z|^{2k+a-2} dzd\xi.
\\&=&c_{k,a} \int_{-\infty}^{+\infty} B_{k,a}(x,z)\left(\int_{-\infty}^{+\infty}  \Delta_{k,a}(y,\xi,z)f(\xi)|\xi|^{2k+a-2} \; d\xi\right).
\end{eqnarray*}
Then one obtain that 
$$\tau^{k,a}_y(f)(z)=\int_{-\infty}^{+\infty}  \Delta_{k,a}(y,\xi,z)f(\xi)|\xi|^{2k+a-2} \; d\xi
=\int_{-\infty}^{+\infty}  f(\xi)d \sigma_{y,z}^{k,a}(\xi) .$$
From this formula and density we can  state the following 
\begin{thm}
The generalized translation operator $\tau^{k,a}_y$,  $y\in  \mathbb{R}$  can be extended  to  a bounded operator on  $L^p(\mathbb{R}, |x|^{2k+a-2}dx)$ for every  $1\leq p\leq  \infty$ and its $L_{p}$-norm is uniformly bounded ( for the variable $y$).

\end{thm}


\begin{thebibliography}{HD}
\bibitem{B} H. Bateman, \textit{ Higher Transcendental Functions},  Vol.1 1953. McGraw-Hill, New York.

 \bibitem{BKO} S. Ben Said, T. Kobayashi, B. Orsted,\textit{ Laguerre semigroup and Dunkl operators}, Compos.Math. 148 (2012), 1265–1336.
\bibitem{Luc} S. Ben Said, L. Deleaval, \textit{Translation Operator and Maximal Function for the $(k, 1)$-Generalized Fourier Transform}, Journal
of Functional Analysis, vol. 279, no. 8 (2020), 1-32.
\bibitem{MSM}  M. A. Boubatra, S. Negzaoui,  M. Sifi, \textit{A new product formula involving Bessel functions}, Integral Transforms and Special Functions, (33)2022 , 247-263.
\bibitem{Hand} I.S. Gradshteyn and I.M. Ryzhik, \textit{Table of Integrals, Series, and Products}, Seventh Edition,   Academic Press (2007).
\bibitem{Mac} H. M. Macdonald,\textit{ Note on the evaluation of a certain integral
containing Bessel's functions.}, Proc. London Math. Soc.,  Volumes2-7, Issue 1, 1909,  142-149.

\bibitem{WM1} W. Magnus, F. Oberhettinger, F. G. Tricomi, \textit{tables of integral transforms,} Volume II. 1954.
\bibitem{WM}  W. Magnus, F. Oberhettinger,  R. P. Soni,  \textit{Formulas and Theorems for the Special Functions of Mathematical Physics}. Springer, Berlin (1966)
  \bibitem{Nist} NIST Handbook of Mathematical Functions, edited by Frank W.J. Olver, Daniel W. Lozier, Ronald F. Boisvert, Charles W. Clark
Cambridge Univ. Press, (2010 )
  \bibitem{Tit} E. C. Titchmarsh,  \textit{Introduction to the Theory of Fourier Integrals},  Oxford
University Press, Amen House, London, 1948.

      \bibitem{Wat} G.N. Watson  \textit{A Treatise on the Theory of Bessel Functions}. 2nd Edition, Cambridge University Press, Cambridge.



\end{thebibliography}
\end{document}